\documentclass[11pt]{article}

\usepackage[margin=1.12in]{geometry}
\usepackage{amsmath,amssymb,amsthm,mathtools}
\usepackage{bm}
\usepackage{booktabs}
\usepackage{enumitem}
\usepackage{microtype}
\usepackage{xcolor}
\usepackage{aliascnt}
\usepackage[numbers,sort&compress]{natbib}
\usepackage[colorlinks=true,linkcolor=blue!60!black,citecolor=blue!60!black,urlcolor=blue!60!black]{hyperref}
\usepackage[nameinlink,capitalise,noabbrev]{cleveref}

\newtheorem{theorem}{Theorem}[section]
\newaliascnt{proposition}{theorem}
\newtheorem{proposition}[proposition]{Proposition}
\aliascntresetthe{proposition}
\newaliascnt{lemma}{theorem}
\newtheorem{lemma}[lemma]{Lemma}
\aliascntresetthe{lemma}
\newaliascnt{corollary}{theorem}
\newtheorem{corollary}[corollary]{Corollary}
\aliascntresetthe{corollary}
\newaliascnt{assumption}{theorem}
\newtheorem{assumption}[assumption]{Assumption}
\aliascntresetthe{assumption}
\theoremstyle{definition}
\newaliascnt{definition}{theorem}

\aliascntresetthe{definition}
\theoremstyle{remark}
\newaliascnt{remark}{theorem}
\newtheorem{remark}[remark]{Remark}
\aliascntresetthe{remark}

\crefname{assumption}{assumption}{assumptions}
\Crefname{assumption}{Assumption}{Assumptions}

\newcommand{\R}{\mathbb{R}}
\newcommand{\N}{\mathbb{N}}
\newcommand{\Var}{\operatorname{Var}}

\newcommand{\1}{\mathbf{1}}
\newcommand{\dd}{\,\mathrm{d}}
\newcommand{\cF}{\mathcal{F}}

\newcommand{\cI}{\mathcal{I}}
\newcommand{\cM}{\mathcal{M}}
\newcommand{\cX}{\mathcal{X}}

\newcommand{\bracket}[1]{\left\langle #1\right\rangle}
\newcommand{\norm}[1]{\left\lVert #1\right\rVert}
\newcommand{\abs}[1]{\left\lvert #1\right\rvert}

\title{Moderate Deviations for Nonlinear Hawkes Processes}
\author{
Yingli Wang\thanks{School of Mathematical Sciences, Fudan University, Shanghai, People's Republic of China; \texttt{yingliwang@fudan.edu.cn}}
\and
Lingjiong Zhu\thanks{Department of Mathematics, Florida State University, Tallahassee, Florida, United States of America; \texttt{zhu@math.fsu.edu}}}
\date{\today}

\begin{document}
\maketitle

\begin{abstract}
A Hawkes process is a simple point process whose intensity depends on its
history; the resulting dynamics are generally non-Markovian.
We establish a sample-path moderate deviation principle for a nonlinear
Hawkes process in the full moderate regime.  Since a Poisson cluster
representation is unavailable for nonlinear Hawkes processes, we use the past
configuration as a Markov state and construct a potential, or Poisson corrector,
for the centered stochastic intensity.  A monotone Poisson coupling
shows that the add-one increment of the corrector is uniformly bounded.  The
centered counting process is consequently the sum of a martingale with bounded
jumps and an exponentially negligible boundary term.  Exponential stabilization
of the predictable quadratic variation follows from the process-level large
deviation principle for nonlinear Hawkes processes.  The martingale moderate
deviation theorem then yields the result for every scale between the central-limit
and large-deviation scales.  The same construction gives a response
formula for the asymptotic variance and, in particular, verifies that the variance
dominates the stationary mean intensity in the self-exciting case.
\end{abstract}

\medskip
\noindent\textbf{Keywords.}
Nonlinear Hawkes process; moderate deviations; martingale approximation;
Poisson equation; point process; process-level large deviations.

\medskip
\noindent\textbf{MSC 2020.}
Primary 60F10; secondary 60G55, 60G44.

\section{Introduction and main results}\label{sec:intro-main}

\subsection{Hawkes processes}\label{sec:intro-hawkes}

Let $N$ be a simple point process on $\R$, and let its natural filtration be
\begin{equation*}
  \cF_t^N
  :=\sigma\left\{N(C):C\in\mathcal B(\R),\ C\subset(-\infty,t]\right\}.
\end{equation*}
For $t\ge0$, write $N_t:=N(0,t]$.
A nonnegative predictable process $(\lambda_t)_{t\in\R}$ is called the
stochastic intensity of $N$ if, for every $a<b$ for which the expectations are
finite,
\begin{equation}\label{eq:intensity-definition-intro}
  \mathbb E\!\left[N(a,b]\mid\cF_a^N\right]
  =\mathbb E\!\left[\int_a^b\lambda_s\dd s\,\middle|\,\cF_a^N\right].
\end{equation}
Thus $\lambda_t\dd t$ is the conditional infinitesimal rate of an event at time
$t$, given the past of the process.

The classical linear Hawkes process, introduced in \citep{Hawkes1971}, is the
simple point process with intensity
\begin{equation}\label{eq:linear-intro}
  \lambda_t
  =\nu+\int_{(-\infty,t)}h(t-s)N(\dd s),
  \qquad \nu>0,
\end{equation}
where $h:\R_+\to\R_+$ is a locally integrable \textit{kernel function} (or exciting function) 
and $\nu>0$ is the \textit{baseline intensity}.  Each point at
time $s$ contributes $h(t-s)$ to the future intensity, which gives the process
its self-exciting character.  In the subcritical regime
$\int_0^\infty h(s)\dd s<1$, the immigration-birth representation of
\citep{HawkesOakes1974} realizes the process as a Poisson cluster process and
makes many asymptotic questions accessible through branching-process methods.

A nonlinear Hawkes process replaces the affine dependence in
\eqref{eq:linear-intro} by a general rate function
$\phi:\R_+\to(0,\infty)$:
\begin{equation}\label{eq:intensity-intro}
  \lambda_t=\phi(Z_{t-}),
  \qquad
  Z_{t-}:=\int_{(-\infty,t)}h(t-s)N(\dd s).
\end{equation}
The interval $(-\infty,t)$ excludes a possible point at $t$ and makes the
intensity predictable.  The linear model is recovered by taking
$\phi(z)=\nu+z$; a positive linear coefficient can be absorbed into $h$.
The functions $h$ and $\phi$ are referred to as the \textit{kernel function} (or exciting function) and
the \textit{rate function}, respectively.  The nonlinear formulation and its basic
stability theory were developed in \citep{BremMass1996}.  In this paper, $h$ is
nonnegative and $\phi$ is nondecreasing, so the model remains self-exciting;
the precise assumptions are stated below.  The cluster representation is
generally unavailable in this nonlinear setting, and the dependence on the
entire history remains a central technical difficulty.  For a unified survey of
linear and nonlinear Hawkes processes and their limit theorems, see, for
example, \citep{Zhu2013Thesis}.
The Hawkes process generalizes the Poisson process by incorporating
self-excitation and clustering, making it a versatile model with applications
in neuroscience, genomics, criminology, social networks, healthcare,
seismology, insurance, finance, and machine learning, among many other fields.
For surveys of these applications, see
\citep{Zhu2013Thesis,HawkesBook2021}.

\subsection{Large and Moderate Deviation Principles}
\label{sec:intro-deviations}

We first recall the basic large-deviation terminology.  Let $(P_n)_{n\ge1}$ be
probability measures on a topological space $E$, and let $v_n\to\infty$.  The
sequence $(P_n)$ satisfies a \textit{large deviation principle} (LDP) with speed $v_n$
and rate function $I:E\to[0,\infty]$ if $I$ is lower semicontinuous and, for
every Borel set $A\subset E$,
\begin{equation}\label{eq:ldp-definition}
  -\inf_{x\in A^\circ}I(x)
  \le \liminf_{n\to\infty}\frac{1}{v_n}\log P_n(A)
  \le \limsup_{n\to\infty}\frac{1}{v_n}\log P_n(A)
  \le -\inf_{x\in\overline A}I(x).
\end{equation}
The rate function is called good if every level set
$\{x\in E:I(x)\le c\}$ is compact.  Thus an LDP quantifies exponentially rare
events and identifies their least exponential cost.  We refer to the
foundational monograph \citep{Varadhan1984} and the systematic treatment in
\citep{DemboZeitouni2010} for the general theory and its applications.

A \textit{moderate deviation principle} (MDP) is an LDP for centered fluctuations on a
scale strictly between the central-limit and large-deviation scales.  For
example, let $S_n=\xi_1+\cdots+\xi_n$, where the $\xi_i$ are centered
independent and identically distributed random variables with variance $\sigma^2$
and satisfy a suitable exponential-moment condition.  If
\begin{equation}\label{eq:classical-moderate-scale}
  \sqrt n\ll b_n\ll n,
\end{equation}
then the laws of $S_n/b_n$ satisfy an LDP with speed $b_n^2/n$ and the Gaussian
quadratic rate function $x\mapsto x^2/(2\sigma^2)$; see
\citep[Section~3.7]{DemboZeitouni2010}.  The word ``moderate'' refers precisely
to this intermediate normalization: the deviations are larger than those seen
by the CLT but smaller than the order-$n$ deviations governed by Cram\'er's
theorem \citep{DemboZeitouni2010}.

For linear Hawkes processes, \citep{BordenaveTorrisi2007} proves
a large deviation principle with an explicit rate function,
and \citep{Karim2025} extends the result to the multivariate setting with random marks; see also the precise large-time deviation asymptotics in
\citep{GaoZhu2021Precise}.
For nonlinear Hawkes processes, \citep{Zhu2014Level3} proves a
process-level (also known as Level-3) LDP \citep{DonskerVaradhan1983} for the empirical process of time shifts, with
speed $t$ and a specific-entropy rate function.  By the contraction principle \citep{DemboZeitouni2010}, this leads to
a Level-1 LDP for $N_t/t$.  For Markovian nonlinear Hawkes processes with
exponential kernels or finite sums of exponential kernels,
\citep{Zhu2015MarkovLDP} develops an alternative spectral and variational
characterization of the Level-1 rate function.

At the central-limit scale, \citep{HawkesOakes1974} established a central
limit theorem (CLT) for linear Hawkes processes using the immigration-birth
representation; their result also implies the law of large numbers (LLN).
In the multivariate setting, \citep{BacryEtAl2013} established functional
laws of large numbers and central limit theorems with an explicit long-run
variance.  For nonlinear Hawkes processes, the law of large numbers follows
from the ergodic theorem and the ergodicity established in
\citep{BremMass1996}, while \citep{Zhu2013CLT} proved the functional CLT
\begin{equation}\label{eq:known-clt}
  \left\{\frac{N(0,tu]-\mu tu}{\sqrt{t}}:0\le u\le1\right\}
  \Longrightarrow \{\sigma B_u:0\le u\le1\},
\end{equation}
where $\mu$ is the stationary mean intensity and $\sigma^2$ is the long-run
variance.  Accordingly, the moderate scaling in the present setting is given
by a positive sequence $b_t$ satisfying
\begin{equation}\label{eq:moderate-scale}
  \frac{b_t}{\sqrt{t}}\longrightarrow\infty,
  \qquad
  \frac{b_t}{t}\longrightarrow0,
\end{equation}
as $t\rightarrow\infty$.
For this scaling, one expects $(N_t-\mu t)/b_t$ to satisfy an LDP with speed $b_t^2/t$ and the
quadratic rate inherited from the central limit theorem.  The scalar MDP was
proved for linear Hawkes processes in \citep{Zhu2013MDP}, while a functional
version was obtained through the general theory of Poisson cluster processes in
\citep{GaoWang2020}, and the multivariate extension was studied in
\citep{Yao2018}; see also the precise large-time deviation asymptotics in
\citep{GaoZhu2021Precise}.  These arguments
use the Poisson cluster structure of the linear model.  The moderate deviations
for nonlinear Hawkes processes in \citep{GaoZhu2018} instead concern a
simultaneously large rate and small excitation.  At the network level, a large
deviation principle for the mean-field limit and a moderate deviation principle
for its fluctuations were obtained in
\citep{GaoZhu2023MeanFieldLDP,GaoGaoZhu2023MeanFieldMDP}, respectively.  None
of these results gives the large-time MDP for a nonlinear
Hawkes process.

The present paper fills this gap.  We prove a functional MDP for a fixed
nonlinear Hawkes process throughout the full moderate range
$\sqrt t\ll b_t\ll t$, with speed $b_t^2/t$ and the quadratic path-space rate
function determined by the variance in \eqref{eq:known-clt}.  Since the cluster
representation is unavailable, the proof instead uses a martingale-corrector
reduction and the Level-3 LDP of \citep{Zhu2014Level3}.

\begin{table}[t]
  \caption{Summary of large-time limit theorems for Hawkes processes.}
  \label{tab:ET0p01}
  \centering
  \small
  \begin{tabular}{@{}lcc@{}}
    \toprule
    \textbf{Limit theorems} & \textbf{Linear Hawkes processes}
      & \textbf{Nonlinear Hawkes processes} \\
    \midrule
    LLN & \citep{HawkesOakes1974,BacryEtAl2013} & \citep{BremMass1996} \\
    CLT & \citep{HawkesOakes1974,BacryEtAl2013} & \citep{Zhu2013CLT} \\
    LDP & \citep{BordenaveTorrisi2007,Karim2025}
      & \citep{Zhu2014Level3,Zhu2015MarkovLDP} \\
    MDP & \citep{Zhu2013MDP,Yao2018} & Our paper \\
    \bottomrule
  \end{tabular}
\end{table}

Several extensions and alternative asymptotic scalings complement this
CLT--MDP--LDP picture.  For linear Hawkes processes with random marks, a central
limit theorem and a large deviation principle were proved in
\citep{KarabashZhu2015}, and a moderate deviation principle was established in
\citep{Seol2017}.
For Markovian Hawkes processes with an exponential
kernel, the regime of a large initial intensity was treated at the law of large
numbers and fluctuation scales in \citep{GaoZhu2018InitialLimits}, and at the
large-deviation scale in \citep{GaoZhu2018InitialLDP}.  Functional central limit
theorems under a large-baseline scaling, together with applications to
infinite-server queues, were obtained in \citep{GaoZhu2018Queue}.  Moderate and
large deviations for a generalized linear Hawkes model with generation-dependent
exciting functions were studied in \citep{MehrdadZhu2025}.

There are several general MDP criteria for dependent sequences.  The projective
criteria in \citep{DedeckerEtAl2009} yield a functional result for bounded
stationary variables, while the results in \citep{MerlevedePeligradRio2011}
allow unbounded variables with semiexponential tails under quantitative weak dependence.  The
unit increments $N(k-1,k]-\mu$ are unbounded.  Moreover, applying the latter
criterion with the natural exponential tail and exponential coupling bounds gives
only the restricted range $\sqrt n\ll b_n\ll n^{2/3}$, even when the kernel
has an exponential moment.  The full range in
\eqref{eq:moderate-scale} requires an argument that uses the point-process
structure more directly.

Our approach is based on three observations.  First, although the scalar
excitation is generally non-Markovian, the full translated history, viewed as
a state in the infinite-dimensional configuration space $\cX$ defined in
\Cref{sec:coupling}, is a Markov process.
Second, the Poisson equation for the centered intensity has a corrector whose add-one
increment is uniformly bounded by a subcritical response series.  The use of
Poisson equations to construct martingale approximations for additive
functionals of Markov processes is classical; see \citep{GlynnMeyn1996}.  Third,
the resulting martingale has bounded jumps, so the functional MDP in
\citep{Dembo1996} applies as soon as its predictable quadratic variation
stabilizes exponentially.  The last requirement follows by contracting the
process-level LDP in \citep{Zhu2014Level3}.

\subsection{Model and assumptions}\label{sec:model}

Let $N$ be a simple point process on $\R$, adapted to its natural filtration
$(\cF_t)_{t\in\R}$.  For an interval $I$, write $N(I)$ for the number of
points in $I$, and set $N_t:=N(0,t]$ for $t\ge0$.  The process has predictable
intensity
\begin{equation}\label{eq:intensity}
  \lambda_t:=\phi(Z_{t-}),
  \qquad
  Z_{t-}:=\int_{(-\infty,t)}h(t-s)N(\dd s).
\end{equation}

We work under the following assumptions.  They keep the main argument
transparent; their distinct roles and possible relaxations are discussed after
the main results and in \Cref{sec:conclusion}.

\begin{assumption}[Kernel function]\label{ass:kernel}
The kernel function $h:\R_+\to\R_+$ is continuous, nonincreasing, not identically
zero, and
\begin{equation}\label{eq:h-moments}
  \norm{h}_1:=\int_0^\infty h(s)\dd s<\infty,
  \qquad
  m_1(h):=\int_0^\infty s h(s)\dd s<\infty.
\end{equation}
\end{assumption}

\begin{assumption}[Rate function]\label{ass:rate}
The rate function $\phi:\R_+\to(0,\infty)$ is nondecreasing and globally
Lipschitz with constant $L$.  Write
$\underline\phi=\phi(0)>0$, and assume
\begin{equation}\label{eq:subcriticality}
  \rho:=L\norm{h}_1<1.
\end{equation}
In addition, the rate is sublinear:
\begin{equation}\label{eq:rate-sublinear}
  \lim_{z\to\infty}\frac{\phi(z)}{z}=0.
\end{equation}
\end{assumption}

\begin{remark}[Relation to existing assumptions]
The positivity, monotonicity, Lipschitz property, and contraction condition in
\Cref{ass:rate} are the rate-function assumptions used for the nonlinear
functional CLT in \citep{Zhu2013CLT}; the contraction condition
$L\norm{h}_1<1$ is the classical stability condition of
\citep{BremMass1996}.  The sublinear-growth condition
\eqref{eq:rate-sublinear} is the rate assumption used for the process-level
LDP in \citep{Zhu2014Level3}.
\end{remark}

Under \Cref{ass:kernel,ass:rate}, the stability result in
\citep{BremMass1996} gives a unique stationary and ergodic version of $N$.
We write $\mathbb{P}_\pi$ and $\mathbb{E}_\pi$ for its law and expectation, and
\begin{equation}\label{eq:mu}
  \mu=\mathbb{E}_\pi[N(0,1]]=\mathbb{E}_\pi[\lambda_0].
\end{equation}

Let $D[0,1]$ be the space of real c\`adl\`ag functions.  We use the uniform
topology, which is the restriction to $[0,1]$ of the locally uniform topology
in the martingale MDP in \citep{Dembo1996}.

\subsection{Main results}

We can now state the main result of our paper, a functional moderate deviation principle.  It shows that
throughout the full moderate range, the centered counting path retains the
quadratic Brownian rate function, with the nonlinear dependence entering
through the long-run variance $\sigma^2$.

\begin{theorem}[Functional moderate deviations]\label{thm:main-mdp}
Suppose \Cref{ass:kernel,ass:rate} hold and let $N$ be stationary.  Let
$b_t$ satisfy \eqref{eq:moderate-scale}.  Then
\begin{equation}\label{eq:scaled-process}
  Y_t(u)=\frac{N(0,tu]-\mu tu}{b_t},\qquad 0\le u\le1,
\end{equation}
satisfies an LDP on $D[0,1]$, equipped with the uniform topology, with speed
$b_t^2/t$ and good rate function
\begin{equation}\label{eq:rate-function}
  \cI(f)=
  \begin{cases}
    \displaystyle
    \frac{1}{2\sigma^2}\int_0^1 \left(\dot f(s)\right)^2\dd s,
      &  f\in AC[0,1],\ f(0)=0,\\[1.2ex]
    +\infty, & \text{otherwise},
  \end{cases}
\end{equation}
where $0<\sigma^2<\infty$ is the long-run variance,
\begin{equation}\label{eq:long-run-variance}
  \sigma^2
  =\lim_{t\to\infty}\frac{1}{t}\Var_\pi(N_t),
\end{equation}
and $AC[0,1]$ denotes the space of absolutely continuous functions from
$[0,1]$ to $\mathbb{R}$.  The corrector-response representation of
$\sigma^2$ is given in \eqref{eq:variance-formula} of \Cref{thm:variance}
below.
\end{theorem}

The long-run variance appearing in \eqref{eq:rate-function} admits a response
representation.  To state it, let $X_{0-}$ denote the strict-past
configuration at time zero (see \eqref{eq:history-process} for its formal definition), let $\cX^-$ denote the space of admissible
strict-past configurations (see \eqref{X:space:defn} for its formal definition) and, for a history functional $F$, set
$DF(\omega)=F(\omega+\delta_0)-F(\omega)$; the formal constructions are given
in \Cref{sec:coupling,sec:corrector}.

\begin{theorem}[Response formula for the variance]\label{thm:variance}
Under \Cref{ass:kernel,ass:rate}, let $g:\cX\to\R$ be the zero-resolvent
corrector defined in \eqref{eq:corrector}.  The absolute convergence of the
defining integral is established in \Cref{prop:corrector-bounds}.  Then
\begin{equation}\label{eq:Dg-bound-preview}
  0\le Dg(\omega)\le\frac{\rho}{1-\rho}
  \qquad \omega\in\cX^-,
\end{equation}
and
\begin{equation}\label{eq:variance-formula}
  \sigma^2
  =\mathbb{E}_\pi\!\left[
      \lambda_0\left(1+Dg(X_{0-})\right)^2
    \right].
\end{equation}
Consequently, $\sigma^2\ge\mu$.  If
\begin{equation}\label{eq:strict-condition}
  \mathbb{P}_\pi\left(\lambda_0Dg(X_{0-})>0\right)>0,
\end{equation}
then $\sigma^2>\mu$.  In particular, \eqref{eq:strict-condition} holds when
$h\not\equiv0$ and $\phi$ is strictly increasing on $\R_+$.
\end{theorem}

Formula \eqref{eq:variance-formula} has a direct response interpretation: the
unit term represents the event itself, whereas $Dg(X_{0-})$ represents the
expected cumulative increase in future counts caused by that event.  Thus the
comparison with $\mu$ quantifies the variance inflation due to self-excitation;
the linear specialization is discussed in \Cref{rem:linear-consistency}.

The stationary formulation is convenient for the use of the empirical
process.  The following corollary gives the corresponding result for a process
started with no points before time zero.

\begin{corollary}[Empty history]\label{cor:empty}
Under \Cref{ass:kernel,ass:rate}, let $N^\varnothing$ be the nonlinear Hawkes
process with an empty history on $(-\infty,0]$.  With the stationary centering
$\mu t$, the processes
\begin{equation*}
  \left\{
  \frac{N^\varnothing(0,tu]-\mu tu}{b_t}:0\le u\le1
  \right\}
\end{equation*}
satisfy the same MDP as in \Cref{thm:main-mdp}.
\end{corollary}

\begin{remark}
The contraction condition controls the response to an inserted point, whereas
sublinearity is used only in the exponential estimates that pass from the
process-level LDP to stabilization of the bracket.  In particular,
\Cref{ass:rate} allows bounded saturating rates as well as unbounded sublinear
rates.  The corrector construction and bounded-jump martingale reduction remain
valid under monotonicity, Lipschitz continuity, and $L\norm h_1<1$ without
sublinearity; what then remains missing is the bracket-stabilization argument.
Rates with positive asymptotic slope, including the classical linear rate, are
therefore not covered by \Cref{thm:main-mdp}.  The linear case can instead be
treated through its Poisson cluster representation; see
\Cref{rem:linear-consistency}.
\end{remark}

\begin{remark}[Linear-model consistency]\label{rem:linear-consistency}
Although the sublinear-growth assumption excludes the classical linear model,
the corrector formula has the correct formal specialization.  Let
$\phi(z)=\nu+z$ and $\rho=\norm h_1<1$.  With $H$ and $V$ defined in
\eqref{eq:H-tail} and \eqref{eq:past-influence}, respectively, the
zero-resolvent corrector can be computed explicitly as
\begin{equation*}
  g(\omega)
  =\frac{V(\omega)-\mu m_1(h)}{1-\rho},
  \qquad
  \mu=\frac{\nu}{1-\rho}.
\end{equation*}
Indeed, this follows by integrating the linear renewal equation for
$\mathbb E_\omega[\lambda_t]-\mu$.  Adding a point at the origin increases
$V(\omega)$ by
\begin{equation*}
  H(0)=\int_0^\infty h(s)\dd s=\rho.
\end{equation*}
Consequently,
\begin{equation*}
  Dg(\omega)=\frac{\rho}{1-\rho}=\rho+\rho^2+\cdots.
\end{equation*}
Equivalently, the inserted point has an expected number $\rho^n$ of descendants
in generation $n$, so $Dg$ is the expected total number of its descendants.
Hence $1+Dg=(1-\rho)^{-1}$, while $\mu=\nu/(1-\rho)$, and
\eqref{eq:variance-formula} becomes
\begin{equation*}
  \sigma^2=\frac{\mu}{(1-\rho)^2}
  =\frac{\nu}{(1-\rho)^3}.
\end{equation*}
This agrees with the classical (functional) CLT and and MDP for the linear Hawkes processes; see
\citep{HawkesOakes1974,BacryEtAl2013,Zhu2013MDP}.
\end{remark}

The remainder of the paper is organized as follows.
\Cref{sec:martingale-reduction} develops the response coupling, constructs the
Poisson corrector, and reduces the centered counting process to a bounded-jump
martingale up to an exponentially negligible boundary.
\Cref{sec:bracket} derives exponential stabilization of the bracket from the
process-level LDP and proves the main results.  \Cref{sec:conclusion} summarizes
the argument and records the remaining directions.

\section{Martingale-corrector reduction}\label{sec:martingale-reduction}

\subsection{The history process and response coupling}\label{sec:coupling}

We describe the Markov state and collect the coupling estimates needed later.
The state records the history immediately after a possible event at the current
time.  Thus its configurations live on $(-\infty,0]$, while the atom at the
origin is excluded from the predictable excitation.  For such a configuration
$\omega$, define
\begin{equation}\label{eq:excitation-functional}
  z(\omega):=\int_{(-\infty,0)}h(-s)\omega(\dd s),
  \qquad
  \ell(\omega):=\phi(z(\omega)).
\end{equation}
Let
\begin{equation}\label{eq:H-tail}
  H(a):=\int_a^\infty h(s)\dd s,
  \qquad a\ge0.
\end{equation}
By Fubini's theorem,
\begin{equation}\label{eq:H-integral}
  \int_0^\infty H(a)\dd a=m_1(h)<\infty.
\end{equation}
We take $\cX$ to be the space of locally finite simple point configurations
$\omega$ on $(-\infty,0]$, equipped with the vague topology and satisfying
\begin{equation}\label{eq:history-space}
  z(\omega)<\infty,
  \qquad
  \int_{(-\infty,0]}H(-s)\omega(\dd s)<\infty.
\end{equation}
We write
\begin{align}\label{X:space:defn}
\cX^-:=\left\{\omega\in\cX:\omega(\{0\})=0\right\},
\end{align}
for the
predictable states, to which the add-one operator below is applied.
The post-event and predictable history states at time $t$ are, respectively,
\begin{align}\label{eq:history-process}
  X_t:=\sum_{\tau\in N:\,\tau\le t}\delta_{\tau-t},
  \qquad
  X_{t-}:=\sum_{\tau\in N:\,\tau<t}\delta_{\tau-t}.
\end{align}
In particular, $\lambda_t=\ell(X_{t-})$.  At an event time $t$,
$X_t=X_{t-}+\delta_0$; away from event times the two states agree.
Between events, every point moves to the left at unit speed; at an event, a point
is inserted at the origin.  Thus $(X_t)_{t\ge0}$ is a piecewise deterministic
Markov process on $\cX$ in the sense of \citep{Davis1984}, even when $h$ is not
a sum of exponentials.

Following the Poisson-embedding construction of
\citep[Section~3]{BremMass1996}, let $\Pi$ be a Poisson random measure on
$\R\times\R_+$ with intensity $\dd t\dd z$.  Predictably thinning $\Pi$ below
the graph of $t\mapsto\lambda_t$ gives
\begin{equation}\label{eq:poisson-embedding}
  N(\dd t)=\int_{\R_+}\1_{\{z\le\lambda_t\}}\Pi(\dd t,\dd z).
\end{equation}
Indeed, since $\lambda_t$ is predictable, the compensator of the right-hand
side is $\lambda_t\dd t$, so the resulting point process has the prescribed
stochastic intensity.
Because $\phi$ is nondecreasing and $h$ is nonnegative,
two ordered histories can be coupled monotonically using the same $\Pi$.

Define the response resolvent
\begin{equation}\label{eq:resolvent}
  r(t):=\sum_{n=1}^\infty L^n h^{*n}(t),\qquad t\ge0,
\end{equation}
where $h^{*n}$ is the $n$-fold convolution on $\R_+$.  Subcriticality condition $\rho<1$ gives
\begin{equation}\label{eq:resolvent-moments}
  \int_0^\infty r(t)\dd t=\frac{\rho}{1-\rho},
  \qquad
  \int_0^\infty t\cdot r(t)\dd t
  =\frac{L m_1(h)}{(1-\rho)^2}.
\end{equation}
Indeed, since $h$ is nonnegative and integrable,
\begin{equation}\label{two:identities}
  \int_0^\infty h^{*n}(s)\dd s=\norm h_1^n,
  \qquad
  \int_0^\infty s\cdot h^{*n}(s)\dd s
  =n\norm h_1^{n-1}m_1(h).
\end{equation}
Summing the first identity in \eqref{two:identities} against $L^n$ gives
\[
\sum_{n=1}^{\infty}\rho^n=\frac{\rho}{1-\rho}.
\]
Similarly, summing the second identity in \eqref{two:identities} gives
\[
Lm_1(h)\sum_{n=1}^{\infty}n\rho^{n-1}=\frac{Lm_1(h)}{(1-\rho)^2}.
\]
Thus both $r\in L^1(\R_+)$ and its first moment are finite.

The resolvent $r$ sums the direct response to an inserted point and all
subsequent generations of discrepancies.  The next lemma makes this
interpretation precise and gives the uniform total-response bound used later
to control the add-one increment of the corrector.

\begin{lemma}[Response to one inserted point]\label{lem:add-one-response}
Let $N^+$ and $N$ be coupled by \eqref{eq:poisson-embedding}, with the same
history except that the history of $N^+$ contains one additional point at time
zero.  Let $D=N^+-N$ be the discrepancy process on $(0,\infty)$.  Then
\begin{equation}\label{eq:response-density}
  \frac{\dd}{\dd t}\mathbb{E}[D(0,t]]\le r(t)
  \quad\text{for almost every $t>0$},
\end{equation}
and therefore,
\begin{equation}\label{eq:response-total}
  \mathbb{E}[D(0,\infty)]\le\frac{\rho}{1-\rho}.
\end{equation}
More generally, if the additional point has age $a\ge0$ at time zero, then
\begin{equation}\label{eq:aged-response}
  \mathbb{E}[D(0,\infty)]
  \le \frac{L}{1-\rho}H(a).
\end{equation}
\end{lemma}

\begin{proof}
The direct intensity difference caused by a point inserted at time zero
is bounded by $Lh(t)$.  Every discrepancy point at time $s$ can in turn create
an intensity discrepancy bounded by $Lh(t-s)$.  Iterating the Poisson embedding
construction, as in the stability coupling of \citep{BremMass1996}, yields
\begin{equation*}
  \frac{\dd}{\dd t}\mathbb{E}[D(0,t]]
  \le Lh(t)+L^2h^{*2}(t)+\cdots=r(t).
\end{equation*}
Integration and \eqref{eq:resolvent-moments} give
\eqref{eq:response-total}.  For a point of age $a$, the direct source is
$Lh(a+t)$.  The total mass of this source is $LH(a)$, and every generation
after the first multiplies total mass by at most $\rho$.  Summing the geometric
series proves \eqref{eq:aged-response}.
\end{proof}

The exponential moment estimate in the next lemma will control the corrector boundary.  Since
$\phi(z)\le\underline\phi+Lz$, the nonlinear process is dominated,
under the monotone Poisson coupling, by the stable linear Hawkes process with
immigration rate $\underline\phi$ and reproduction kernel $Lh$.

\begin{lemma}[Past influence has exponential moments]\label{lem:past-exp}
Set
\begin{equation}\label{eq:past-influence}
  V(\omega):=\int_{(-\infty,0]}H(-s)\omega(\dd s).
\end{equation}
If $w:\R_+\to\R_+$ is bounded and integrable, then for some
$\theta_w>0$,
\begin{equation}\label{eq:V-exp}
  \mathbb{E}_\pi\!\left[
    \exp\left\{\theta_w\int_{(-\infty,0]}w(-s)N(\dd s)\right\}
  \right]<\infty.
\end{equation}
The same estimate holds uniformly in time for the process started with an empty
history.  In particular, $V(X_t)$, $z(X_{t-})$, and the number of points in
any interval of length one have exponential moments in a neighborhood of zero,
uniformly in $t$, under both initial laws.
\end{lemma}

\begin{proof}
Let $\bar N$ be the stationary linear Hawkes process with intensity
\begin{equation*}
  \bar\lambda_t:=\underline\phi+
    L\int_{(-\infty,t)}h(t-s)\bar N(\dd s).
\end{equation*}
Its branching ratio is $\rho<1$.  Monotonicity and
$\phi(z)\le\underline\phi+Lz$ allow the stationary nonlinear process,
and also every empty-history version, to be coupled below $\bar N$.

Use the Poisson cluster representation of $\bar N$ from
\citep{HawkesOakes1974}.  If $S$ is the total progeny of one immigrant, then $S$
is the total population of a subcritical Poisson Galton--Watson process and
therefore
$\mathbb{E}[e^{\theta S}]<\infty$ for all sufficiently small $\theta>0$.  For an
immigrant at time $-a$, let
\begin{equation*}
  W_a:=\sum_{j:T_j\le a}w(a-T_j),
\end{equation*}
where $(T_j)$ are the relative times of the points in its cluster.  Since
$W_a\le \norm w_\infty S$,
\begin{equation*}
  e^{\theta W_a}-1
  \le \theta W_a e^{\theta\norm w_\infty S},
  \qquad
  \int_0^\infty\mathbb{E}\left[W_a e^{\theta\norm w_\infty S}\right]\dd a
  =\norm w_1\mathbb{E}\left[S e^{\theta\norm w_\infty S}\right]<\infty,
\end{equation*}
for small $\theta$.  The Laplace functional of the immigrant Poisson process
\citep{DaleyVereJones2003} now proves \eqref{eq:V-exp}.  By
\Cref{ass:kernel}, the function $h$ is bounded, while
\begin{equation*}
  \norm H_\infty\le\norm h_1,
  \qquad
  \norm H_1=m_1(h)<\infty.
\end{equation*}
Hence, the preceding calculation applies to
$w=h$, $w=H$, and $w=\1_{[0,1]}$.  Taking $\theta>0$ sufficiently small, the monotone coupling with the stationary
linear process $\bar N$ and its stationarity give, under either initial law
$\mathbb P_*\in\{\mathbb P_\pi,\mathbb P_\varnothing\}$,
\begin{align*}
  \sup_{t\ge0}\mathbb E_*
  \exp\left\{\theta V(X_t)\right\}
  &\le
  \mathbb E\exp\left\{
    \theta\int_{(-\infty,0]}H(-s)\bar N(\dd s)
  \right\}<\infty,\\
  \sup_{t\ge0}\mathbb E_*
  \exp\left\{\theta z(X_{t-})\right\}
  &\le
  \mathbb E\exp\left\{
    \theta\int_{(-\infty,0]}h(-s)\bar N(\dd s)
  \right\}<\infty,\\
  \sup_{t\ge0}\mathbb E_*
  \exp\left\{\theta N(t,t+1]\right\}
  &\le
  \mathbb E\exp\left\{\theta\bar N(0,1]\right\}<\infty.
\end{align*}
These are the claimed uniform exponential-moment estimates.
\end{proof}

To transfer the moderate deviation result between the two initial laws, we
also need quantitative control of the points generated by the stationary
prehistory.  The monotone embedding and the preceding exponential-moment
bound yield the exponentially integrable coupling estimate in the following lemma.

\begin{lemma}[Stationary--empty coupling]
\label{lem:stationary-empty-coupling}
Couple a stationary process $N^\pi$ and an empty-history process
$N^\varnothing$ with the same Poisson embedding.  Then
$N^\varnothing\le N^\pi$, and the total discrepancy
\begin{equation}\label{eq:total-coupling-discrepancy}
  D_\infty=N^\pi(0,\infty)-N^\varnothing(0,\infty)
\end{equation}
has a finite exponential moment for some positive parameter.
\end{lemma}

\begin{proof}
Conditional on the stationary past, the first generation of discrepancy points
is dominated by a Poisson variable with random mean
\begin{equation*}
  B_0:=L\int_{(-\infty,0)}H(-s)N(\dd s)=LV(X_0).
\end{equation*}
Each discrepancy point produces subsequent discrepancies dominated by a Poisson
Galton--Watson process with offspring mean $\rho<1$.  The total progeny $S$
of this subcritical process has a finite exponential moment for all sufficiently
small positive parameters.  If $K$, conditional on $B_0$, is Poisson with
mean $B_0$, then the total discrepancy is dominated by
$S_1+\cdots+S_K$, and
\begin{equation*}
  \mathbb{E}\left[e^{\theta\sum_{i=1}^K S_i}\mid B_0\right]
  =\exp\left\{B_0\left(\mathbb{E} e^{\theta S}-1\right)\right\}.
\end{equation*}
For small $\theta$, \Cref{lem:past-exp} makes the expectation of the
right-hand side finite, proving the claim.
\end{proof}

\subsection{Poisson corrector and martingale decomposition}\label{sec:corrector}

Let $(P_t)_{t\ge0}$ be the transition semigroup of the history process
$(X_t)_{t\ge0}$ and set
\begin{equation}\label{eq:centered-intensity}
  f(\omega):=\ell(\omega)-\mu.
\end{equation}
The desired corrector is the zero resolvent
\begin{equation}\label{eq:corrector}
  g(\omega):=\int_0^\infty P_s f(\omega)\dd s.
\end{equation}
The coupling estimates below show that this integral is absolutely convergent for
the histories reached under the stationary or empty-history law, and more
generally on $\cX$.

For a measurable functional $F:\cX\to\R$ and $\omega\in\cX^-$, define its
add-one increment by
\begin{equation}\label{eq:add-one}
  DF(\omega):=F(\omega+\delta_0)-F(\omega).
\end{equation}

The preceding response estimates control both the dependence of $g$ on the
remote past and its sensitivity to a newly inserted point.  The former will
make the corrector boundary negligible, while the latter will give uniformly
bounded jumps for the corrected martingale.

\begin{proposition}[Corrector bounds]\label{prop:corrector-bounds}
The integral in \eqref{eq:corrector} admits a version such that
\begin{equation}\label{eq:g-growth}
  \abs{g(\omega)}
  \le C_0+C_1V(\omega),
  \qquad \omega\in\cX,
\end{equation}
for finite constants $C_0,C_1$, and
\begin{equation}\label{eq:Dg-bound}
  0\le Dg(\omega)\le\frac{\rho}{1-\rho},
  \qquad \omega\in\cX^-.
\end{equation}
In particular, there exists $\theta_g>0$ such that
\begin{equation}\label{eq:g-uniform-exp}
    \sup_{t\ge0}\mathbb E_\pi
      \left[e^{\theta_g\abs{g(X_t)}}\right]<\infty,
      \quad\text{and}
\quad\sup_{t\ge0}\mathbb E_\varnothing
      \left[e^{\theta_g\abs{g(X_t)}}\right]
  <\infty.
\end{equation}
\end{proposition}

\begin{proof}
Couple processes started from $\omega+\delta_0$ and $\omega$.  Monotonicity
and the compensator identity give
\begin{align*}
  Dg(\omega)
  =\int_0^\infty
    \left(\mathbb{E}_{\omega+\delta_0}[\lambda_s]
          -\mathbb{E}_\omega[\lambda_s]\right)\dd s
          =\mathbb{E}\left[N^+(0,\infty)-N(0,\infty)\right].
\end{align*}
The quantity is nonnegative and is bounded by \eqref{eq:response-total}.

For the growth estimate, first compare a finite truncation of $\omega$ to the
empty history one point at a time and then pass to the full history by the
monotone convergence theorem.  The contribution of a past point of age $a$ to the integrated
future intensity difference is bounded by $LH(a)/(1-\rho)$, by
\eqref{eq:aged-response}.  Therefore,
\begin{equation}\label{eq:g-difference-empty}
  \abs{g(\omega)-g(\varnothing)}
  \le\frac{L}{1-\rho}V(\omega).
\end{equation}
To justify both the absolute convergence in \eqref{eq:corrector} and the
finiteness of $g(\varnothing)$, couple the empty-history process below a
stationary process.  The integral over time of their expected intensity
difference is bounded by
$L\mathbb{E}_\pi[V(X_0)]/(1-\rho)$.  This is finite by
\Cref{lem:past-exp}.  The same comparison, followed by the preceding
point-by-point bound, proves absolute convergence for every $\omega\in\cX$.
The exponential moment conclusion follows from \eqref{eq:g-growth} and
\Cref{lem:past-exp}.
\end{proof}

We next put $g$ in the extended generator domain.  This can be done without
assuming differentiability of $h$.  For $\varepsilon>0$, define the resolvent
\begin{equation}\label{eq:epsilon-resolvent}
  g_\varepsilon(\omega)
  :=\int_0^\infty e^{-\varepsilon s}P_s f(\omega)\dd s.
\end{equation}

The discounted resolvents provide approximations to the zero resolvent that
belong to the extended generator domain.  The next lemma gives the uniform
bounds and convergence needed to pass their Dynkin martingale representations
to the limit.

\begin{lemma}[Resolvent convergence and representation]
\label{lem:resolvent-convergence}
For $0<\varepsilon\le1$, the resolvents satisfy, with the same type of
constants as in \eqref{eq:g-growth},
\begin{equation}\label{eq:gepsilon-growth}
  \abs{g_\varepsilon(\omega)}\le C_0+C_1V(\omega),
  \qquad \omega\in\cX,
\end{equation}
uniformly in $\varepsilon$, and
$g_\varepsilon(\omega)\to g(\omega)$ for every $\omega\in\cX$ as
$\varepsilon\downarrow0$.  Moreover,
\begin{equation}\label{eq:Dgepsilon-convergence}
\begin{aligned}
 & 0\le Dg_\varepsilon(\omega)\le\frac{\rho}{1-\rho},
  \qquad \omega\in\cX^-,
\\
 & \sup_{\omega\in\cX^-}\abs{Dg_\varepsilon(\omega)-Dg(\omega)}
  \le\int_0^\infty(1-e^{-\varepsilon s})r(s)\dd s
  \rightarrow 0,\qquad\text{as $\varepsilon\downarrow0$}.
\end{aligned}
\end{equation}
Under either the stationary law or the empty-history law, as
$\varepsilon\downarrow0$, for each $T<\infty$,
\begin{align}
  &g_\varepsilon(X_T)\longrightarrow g(X_T),
    \qquad\text{in $L^2$},\label{eq:gepsilon-endpoint-L2}\\
  &\int_0^T\abs{g_\varepsilon(X_s)-g(X_s)}^2\dd s
    \longrightarrow0,
    \qquad\text{in $L^1$}.\label{eq:gepsilon-path-L2}
\end{align}
Finally, the square-integrable Dynkin martingale associated with
$g_\varepsilon$ has the representation
\begin{equation}\label{eq:gepsilon-martingale-representation}
\begin{aligned}
  M^{g,\varepsilon}_t
  &:=g_\varepsilon(X_t)-g_\varepsilon(X_0)
    -\int_0^t\left(\varepsilon g_\varepsilon(X_s)-f(X_s)\right)\dd s
  \\
  &=\int_{(0,t]}Dg_\varepsilon(X_{s-})
    \left(N(\dd s)-\lambda_s\dd s\right).
\end{aligned}
\end{equation}
\end{lemma}

\begin{proof}
The comparison used in the proof of \Cref{prop:corrector-bounds}, with the
additional factor $e^{-\varepsilon s}\le1$, gives
\begin{equation*}
  \abs{g_\varepsilon(\omega)-g_\varepsilon(\varnothing)}
  \le\frac{L}{1-\rho}V(\omega).
\end{equation*}
Coupling the empty-history process below a stationary process similarly gives
\begin{equation*}
  \sup_{0<\varepsilon\le1}
  \abs{g_\varepsilon(\varnothing)}
  \le\frac{L}{1-\rho}\mathbb{E}_\pi[V(X_0)].
\end{equation*}
This proves \eqref{eq:gepsilon-growth}.  Absolute convergence of
\eqref{eq:corrector} and dominated convergence in the time integral give the
pointwise convergence to $g$.

For $\omega\in\cX^-$, the monotone coupling and compensator identity give
\begin{equation*}
  Dg_\varepsilon(\omega)
  =\int_0^\infty e^{-\varepsilon s}
    \left(\mathbb{E}_{\omega+\delta_0}[\lambda_s]
          -\mathbb{E}_\omega[\lambda_s]\right)\dd s.
\end{equation*}
The intensity difference is nonnegative and is bounded, after integration
against time, by the response density $r$.  This proves both assertions in
\eqref{eq:Dgepsilon-convergence}; the last limit follows from the dominated
convergence theorem because $r\in L^1(\R_+)$.

The uniform growth bound and \Cref{lem:past-exp} provide, under both initial
laws, an integrable dominating function for the squared differences.  The
dominated convergence theorem therefore proves
\eqref{eq:gepsilon-endpoint-L2} and, after integration over
$[0,T]$, \eqref{eq:gepsilon-path-L2}.

Although $f$ is not bounded, it has at most linear growth in $z$, and
\eqref{eq:gepsilon-growth} gives the corresponding bound for the resolvent.
The semigroup property and Fubini's theorem yield, for $t\ge0$,
\begin{equation}\label{eq:resolvent-semigroup-identity}
  P_tg_\varepsilon-g_\varepsilon
  =\int_0^tP_s\left(\varepsilon g_\varepsilon-f\right)\dd s.
\end{equation}
Indeed,
\begin{equation*}
  P_tg_\varepsilon
  =\int_0^\infty e^{-\varepsilon u}P_{t+u}f\dd u
  =e^{\varepsilon t}
    \left(g_\varepsilon-\int_0^t e^{-\varepsilon v}P_vf\dd v\right),
\end{equation*}
and differentiating the last expression in $t$ and integrating from $0$ to
$t$ gives \eqref{eq:resolvent-semigroup-identity}.  All uses of Fubini's theorem are
justified under the stationary and empty-history laws by the growth bounds
above and the exponential moments in \Cref{lem:past-exp}.

Set $a_\varepsilon:=\varepsilon g_\varepsilon-f$.  For $0\le s\le t$, the
Markov property and \eqref{eq:resolvent-semigroup-identity} give
\begin{align*}
  &\mathbb{E}\left[\left.
    g_\varepsilon(X_t)-g_\varepsilon(X_s)
    -\int_s^t a_\varepsilon(X_u)\dd u
    \right|\cF_s\right]\\
  &\qquad=P_{t-s}g_\varepsilon(X_s)-g_\varepsilon(X_s)
    -\int_0^{t-s}P_va_\varepsilon(X_s)\dd v=0.
\end{align*}
Consequently, $\mathcal A g_\varepsilon=\varepsilon g_\varepsilon-f$ in the
extended-generator sense, and the process on the left of
\eqref{eq:gepsilon-martingale-representation} is a square-integrable
martingale.  At an event time $s$, the
state changes from $X_{s-}$ to $X_{s-}+\delta_0$, and hence its jump is
$Dg_\varepsilon(X_{s-})$.  The martingale representation theorem for a simple
point process in its natural filtration, with the initial history included in
$\cF_0$, now yields the stochastic-integral representation; see, for example,
\citep{Jacod1975} and \citep[Chapter III]{JacodShiryaev2003}.  Its square
integrability follows from
\eqref{eq:Dgepsilon-convergence} and
\begin{equation*}
  \mathbb{E}\left[\int_0^T\lambda_s\dd s\right]\le\mu T,
\end{equation*}
where equality holds under stationarity and the inequality follows by coupling
the empty-history process below the stationary process.
\end{proof}

Passing to the limit $\varepsilon\downarrow0$ in the preceding resolvent
representation identifies the martingale associated with the corrector.
Adding it to the compensated counting martingale produces the bounded-jump
martingale that drives the moderate deviation argument.

\begin{proposition}[Martingale-corrector decomposition]\label{prop:decomposition}
Under either the stationary law or the empty-history law, the process
\begin{equation}\label{eq:Mg-integral}
  M^g_t:=\int_{(0,t]}Dg(X_{s-})
    \left(N(\dd s)-\lambda_s\dd s\right)
\end{equation}
is a locally square-integrable, purely discontinuous martingale and satisfies
\begin{equation}\label{eq:g-martingale}
  M^g_t=g(X_t)-g(X_0)+\int_0^t(\lambda_s-\mu)\dd s.
\end{equation}
If
\begin{equation}\label{eq:count-martingale}
  M^N_t=N_t-\int_0^t\lambda_s\dd s,
\end{equation}
and $\widetilde M=M^N+M^g$, then
\begin{equation}\label{eq:main-decomposition}
  N_t-\mu t=\widetilde M_t+g(X_0)-g(X_t).
\end{equation}
Moreover,
\begin{equation}\label{eq:Mtilde-integral}
  \widetilde M_t
  =\int_{(0,t]}\left(1+Dg(X_{s-})\right)
    \left(N(\dd s)-\lambda_s\dd s\right),
\end{equation}
and therefore
\begin{equation}\label{eq:jump-bound}
  0<\Delta\widetilde M_s
  =1+Dg(X_{s-})\le\frac{1}{1-\rho}.
\end{equation}
\end{proposition}

\begin{proof}
By \eqref{eq:Dgepsilon-convergence} and the point-process isometry, for every
$T<\infty$,
\begin{align*}
  \mathbb{E}\left[\abs{M^{g,\varepsilon}_T-M^g_T}^2\right]
  &=\mathbb{E}\left[\int_0^T
      \abs{Dg_\varepsilon(X_{s-})-Dg(X_{s-})}^2\lambda_s\dd s\right]\\
  &\le \mu T
    \sup_{\omega\in\cX^-}\abs{Dg_\varepsilon(\omega)-Dg(\omega)}^2
  \longrightarrow0.
\end{align*}
The uniform bound \eqref{eq:gepsilon-growth} and \Cref{lem:past-exp} also give
\begin{equation*}
  \mathbb{E}\left[\abs{\varepsilon\int_0^T
      g_\varepsilon(X_s)\dd s}^2\right]
  \le \varepsilon^2T\int_0^T\mathbb{E}\left[\left(g_\varepsilon(X_s)\right)^2\right]\dd s
  \le C_T\varepsilon^2
  \rightarrow 0,
\end{equation*}
as $\varepsilon\downarrow0$.
Together with \eqref{eq:gepsilon-endpoint-L2}, this allows passage to the limit
on the left side of
\eqref{eq:gepsilon-martingale-representation}.  Thus
\eqref{eq:g-martingale} holds for each $T$, and hence indistinguishably after
taking the c\`adl\`ag version supplied by \eqref{eq:Mg-integral}.  Combining
\eqref{eq:Mg-integral} with \eqref{eq:count-martingale} proves
\eqref{eq:main-decomposition} and \eqref{eq:Mtilde-integral}.  Finally,
\eqref{eq:Dg-bound} gives \eqref{eq:jump-bound}.
\end{proof}

\subsection{Exponential negligibility of the corrector boundary}

The reduction is completed by showing in the next lemma that the corrector boundary is negligible
on every moderate-deviations scale.

\begin{lemma}[Boundary exponential equivalence]\label{lem:boundary-negligible}
For every $\varepsilon>0$,
\begin{equation}\label{eq:boundary-negligible}
  \limsup_{t\to\infty}\frac{t}{b_t^2}
  \log\mathbb{P}_*\left(
    \sup_{0\le u\le1}\abs{g(X_{tu})-g(X_0)}>\varepsilon b_t
  \right)=-\infty.
\end{equation}
Here, $\mathbb{P}_*$ may be either $\mathbb{P}_\pi$ or $\mathbb{P}_\varnothing$.
\end{lemma}

\begin{proof}
By \eqref{eq:g-growth}, it suffices to control $V(X_s)$ uniformly for
$0\le s\le t$.  On a unit interval, monotonicity of $H$ gives
\begin{equation*}
  \sup_{k\le s\le k+1}V(X_s)
  \le V(X_k)+H(0)N(k,k+1].
\end{equation*}
The two terms on the right are dominated by the corresponding functionals of
the stable linear Hawkes process in \Cref{lem:past-exp}.  Thus there are
constants $c,C>0$ such that
\begin{equation}\label{eq:unit-sup-tail}
  \sup_{k\in\N}\mathbb{P}_*\left(
    \sup_{k\le s\le k+1}V(X_s)>x
  \right)\le Ce^{-cx}.
\end{equation}
Therefore, the probability in \eqref{eq:boundary-negligible} is bounded by
$C(t+1)e^{-c' b_t}$.  Since $b_t/\sqrt t\to\infty$, in particular
$b_t/\log t\to\infty$, and
\begin{equation*}
  \frac{t}{b_t^2}\left(\log(C(t+1))-c'b_t\right)
  \longrightarrow-\infty
\end{equation*}
because $t/b_t\to\infty$.  This proves the claim.
\end{proof}

\section{Exponential stabilization and proof of the MDP}\label{sec:bracket}

The preceding reduction leaves one probabilistic condition to verify for the martingale MDP:
exponential stabilization of its predictable quadratic variation.  We now
derive this property from the process-level LDP.

It follows from \eqref{eq:Mtilde-integral} that
\begin{equation}\label{eq:bracket}
  \bracket{\widetilde M}_t=\int_0^t q(X_{s-})\dd s,
\end{equation}
where
\begin{equation}\label{eq:q-def}
  q(\omega):=\ell(\omega)\left(1+Dg(\omega)\right)^2,
  \qquad \omega\in\cX^-.
\end{equation}
The observable need not be bounded, but it has linear growth in the intensity:
\begin{equation}\label{eq:q-bound}
  \underline\phi\le q(\omega)\le
  \frac{\ell(\omega)}{(1-\rho)^2},
  \qquad
  \mathbb{E}_\pi[q(X_{0-})]\le\frac{\mu}{(1-\rho)^2}<\infty.
\end{equation}

\subsection{Level-3 large deviations and sublinear estimates}

We now establish the hypothesis needed for the bounded-jump martingale MDP.
Let $\Omega$ be the space of locally finite simple point configurations on
$\R$, equipped with the vague topology, and let $\cM_S$ be the space of
stationary probability laws on $\Omega$ with a finite first moment, equipped
with the strengthened weak topology used in
\citep{Zhu2014Level3}: weak convergence together with convergence of
$Q\mapsto\mathbb{E}_Q[N(0,1]]$.  Let $\mathsf H(Q)$ denote the process-level
entropy rate of $Q$ relative to the nonlinear Hawkes dynamics, as defined in
\citep{Zhu2014Level3}.

For $\omega\in\Omega$, write
\begin{equation*}
  p^-\omega:=\omega|_{(-\infty,0)}.
\end{equation*}
For every $Q\in\cM_S$, one has $p^-\omega\in\cX^-$ for $Q$-almost every
$\omega$.  Whenever a functional $F$ defined on $\cX^-$ is integrated
against a law on $\Omega$, we use the same symbol $F$ for the lifted
functional $F\circ p^-$.

For a path $\omega$ observed on $(0,t]$, let $\omega_t$ be its
$t$-periodic extension to $\R$, and let $\theta_s$ denote the time shift by
$s$.  Set
\begin{equation}\label{eq:periodized-empirical-process}
  \mathcal R_{t,\omega}
  :=\frac1t\int_0^t\delta_{\theta_s\omega_t}\dd s.
\end{equation}
When the observed configuration is the restriction
$\omega=N|_{(0,t]}$, we use the shorthand
$\mathcal R_{t,N}:=\mathcal R_{t,N|_{(0,t]}}$.  Thus
$\mathcal R_{t,N}$ is the random empirical process obtained by substituting
the observed point configuration into
\eqref{eq:periodized-empirical-process}.
The empirical measures $\mathcal R_{t,\omega}$ and $\mathcal R_{t,N}$ take
values in $\cM_S$.  Theorem~1 in \citep{Zhu2014Level3} states that,
under the empty-history law $\mathbb{P}_\varnothing$, for every closed
$C\subset\cM_S$,
\begin{equation}\label{eq:empty-level3-upper}
  \limsup_{t\to\infty}\frac1t
  \log\mathbb{P}_\varnothing(\mathcal R_{t,N}\in C)
  \le-\inf_{Q\in C}\mathsf H(Q).
\end{equation}
The assumptions on $h$ and $\phi$, including
\eqref{eq:rate-sublinear}, are precisely of the form required there.  Notice
that the upper bound is stated under the empty-history law; we retain that
initial law throughout this section and return to stationarity by coupling in
\Cref{sec:proof-mdp}.

For $Q\in\cM_S$, define
\begin{equation}\label{eq:Fq}
  \mathsf Q(Q):=\int_\Omega q(\omega)Q(\dd\omega).
\end{equation}
Thus $q(\omega)$ in \eqref{eq:Fq} means that \eqref{eq:q-def} is applied to
the strict-past configuration $p^-\omega$.
Although $q$ depends on the entire past, it can be approximated by bounded
continuous local observables.  We state the precise interface needed below.

We will repeatedly use the superexponential estimates established in the
proof of the Level-3 LDP upper bound
\eqref{eq:empty-level3-upper}; see \citep{Zhu2014Level3}.
The following convenient formulation also records how
sublinearity enters our proof.  If $N_t^{\mathrm{per}}$ denotes the
$t$-periodic extension of the points in $(0,t]$, write
\begin{align*}
  z_s&:=\int_{(0,s)}h(s-u)N(\dd u),\\
  z_s^{\mathrm{per}}
    &:=\int_{(-\infty,s)}h(s-u)N_t^{\mathrm{per}}(\dd u).
\end{align*}
For a nonnegative kernel $w$, also set
\begin{align*}
  J_t(w)&:=\int_0^t\int_{(0,s)}w(s-u)N(\dd u)\dd s,\\
  J_t^{\mathrm{per}}(w)
    &:=\int_0^t\int_{(-\infty,s)}w(s-u)
        N_t^{\mathrm{per}}(\dd u)\dd s.
\end{align*}

The quantities above measure the errors produced by truncating large
excitations, large values of $q$, and remote-past kernel contributions.  The
next lemma shows that these errors are superexponentially small at speed $t$;
this is the step in which the sublinear-growth assumption on $\phi$ enters.

\begin{lemma}[Sublinear superexponential estimates]
\label{lem:sublinear-superexp}
Under $\mathbb{P}_\varnothing$, for every $\varepsilon>0$,
\begin{align}
\lim_{A\to\infty}\limsup_{t\to\infty}\frac1t\log\mathbb{P}_\varnothing
  \left(\frac1t\int_0^t
    \left(\1_{\{z_s>A\}}+\1_{\{z_s^{\mathrm{per}}>A\}}\right)\dd s
    >\varepsilon\right)=-\infty,\label{eq:z-tail-superexp},
\end{align}
and
\begin{align}
\lim_{K\to\infty}\limsup_{t\to\infty}\frac1t\log\mathbb{P}_\varnothing
  \left(\frac1t\left\{
    \int_0^t q(X_{s-})\1_{\{q(X_{s-})>K\}}\dd s
    +t\int_\Omega q\1_{\{q>K\}}\dd\mathcal R_{t,N}
  \right\}>\varepsilon\right)=-\infty.
  \label{eq:q-tail-superexp}
\end{align}
If $w_n$ are nonnegative, bounded, integrable kernels with
$\norm{w_n}_1\to0$, then
\begin{equation}\label{eq:small-kernel-superexp}
  \lim_{n\to\infty}\limsup_{t\to\infty}\frac1t\log\mathbb{P}_\varnothing
  \left(\frac{J_t(w_n)+J_t^{\mathrm{per}}(w_n)}{t}
    >\varepsilon\right)=-\infty.
\end{equation}
\end{lemma}

\begin{proof}
Equation~(3.33) and Lemmas~19--24 of \citep{Zhu2014Level3} give
\begin{equation}\label{eq:count-superexp}
  \lim_{M\to\infty}\limsup_{t\to\infty}\frac1t
  \log\mathbb{P}_\varnothing\left(N(0,t]>Mt\right)=-\infty.
\end{equation}
The remaining statements also follow directly from this estimate.  Indeed,
Fubini's theorem gives
\begin{align*}
  \int_0^t z_s\dd s\le\norm h_1N(0,t],
  \qquad
  \int_0^t z_s^{\mathrm{per}}\dd s=\norm h_1N(0,t],
\end{align*}
and
\begin{align*}
  J_t(w_n)\le\norm{w_n}_1N(0,t],
  \qquad
  J_t^{\mathrm{per}}(w_n)=\norm{w_n}_1N(0,t].
\end{align*}
These inequalities prove \eqref{eq:z-tail-superexp} and
\eqref{eq:small-kernel-superexp}.

Let $c_\rho:=(1-\rho)^{-2}$.  By \eqref{eq:q-bound},
$q\le c_\rho\phi(z)$.  For large $K$, define
\begin{equation*}
  a_K:=\inf\left\{\frac{z}{\phi(z)}:
       \phi(z)>\frac{K}{c_\rho}\right\},
\end{equation*}
with the infimum of the empty set interpreted as $+\infty$.
Sublinearity gives $a_K\to\infty$, and hence
\begin{equation*}
  q\1_{\{q>K\}}
  \le\frac{c_\rho}{a_K}z.
\end{equation*}
Apply the two preceding Fubini identities and
\eqref{eq:count-superexp} to the actual and periodized histories.  This proves
\eqref{eq:q-tail-superexp}.
\end{proof}

We first record a stability estimate for Hawkes processes subject to
deterministic forcing.  For a nonnegative locally integrable forcing profile
$\psi$, let
$N^\psi$ denote the process on $(0,\infty)$ with an empty history and intensity
\begin{equation}\label{eq:forced-intensity}
  \lambda_t^\psi
  :=\phi\!\left(\psi(t)+\int_{(0,t)}h(t-s)N^\psi(\dd s)\right).
\end{equation}
Processes with different forcing profiles are always constructed from the
same Poisson random measure.  For $T<\infty$, set
\begin{equation}\label{eq:GammaT}
  \Gamma_T(\psi)
  :=\int_0^T\left(\mathbb{E}\left[\lambda_s^{\psi+h}\right]
                       -\mathbb{E}\left[\lambda_s^\psi\right]\right)\dd s,
\end{equation}
where $\psi+h$ denotes the forcing profile $t\mapsto\psi(t)+h(t)$.
We have a stability estimate in the following lemma.

\begin{lemma}[Stability under deterministic forcing]
\label{lem:forcing-stability}
For nonnegative locally integrable forcing profiles $\psi,\varphi$,
\begin{equation}\label{eq:forcing-stability}
  \int_0^T\mathbb{E}\abs{\lambda_s^\psi-\lambda_s^\varphi}\dd s
  \le \frac{L}{1-\rho}\int_0^T\abs{\psi(s)-\varphi(s)}\dd s.
\end{equation}
Moreover,
\begin{equation}\label{eq:Gamma-bounds}
  0\le\Gamma_T(\psi)\le\frac{\rho}{1-\rho},
  \qquad
  \abs{\Gamma_T(\psi)-\Gamma_T(\varphi)}
  \le\frac{2L}{1-\rho}
       \int_0^T\abs{\psi(s)-\varphi(s)}\dd s.
\end{equation}
\end{lemma}

\begin{proof}
Under the common Poisson embedding, the expected density of discrepancy points
between $N^\psi$ and $N^\varphi$ equals
$\mathbb{E}\abs{\lambda_t^\psi-\lambda_t^\varphi}$.  The Lipschitz property gives
\begin{equation*}
  \mathbb{E}\abs{\lambda_t^\psi-\lambda_t^\varphi}
  \le L\abs{\psi(t)-\varphi(t)}
     +L\int_0^t h(t-s)
       \mathbb{E}\abs{\lambda_s^\psi-\lambda_s^\varphi}\dd s.
\end{equation*}
Integrating over $[0,T]$, using Tonelli's theorem, and absorbing the
convolution term by $\rho<1$ proves
\eqref{eq:forcing-stability}.  For the first bound, monotonicity gives
$\lambda_t^{\psi+h}\ge\lambda_t^\psi$ under the common coupling.  Set
\begin{equation*}
  d_\psi(t)
  :=\mathbb{E}\left[\lambda_t^{\psi+h}-\lambda_t^\psi\right]\ge0.
\end{equation*}
The Lipschitz property and the Poisson coupling imply
\begin{equation*}
  d_\psi(t)
  \le Lh(t)+L\int_0^t h(t-s)d_\psi(s)\dd s.
\end{equation*}
Iterating this Volterra inequality gives
\begin{equation*}
  d_\psi(t)
  \le\sum_{n=1}^\infty L^nh^{*n}(t)=r(t).
\end{equation*}
Consequently,
\begin{equation*}
  0\le\Gamma_T(\psi)
  =\int_0^T d_\psi(s)\dd s
  \le\int_0^\infty r(s)\dd s
  =\frac{\rho}{1-\rho},
\end{equation*}
which proves the first bound in \eqref{eq:Gamma-bounds}.

For the second bound in \eqref{eq:Gamma-bounds}, insert the corresponding terms with forcing profiles
$\varphi+h$ and $\varphi$.  By the triangle inequality,
\begin{align*}
  \abs{\Gamma_T(\psi)-\Gamma_T(\varphi)}
  \le
  \int_0^T
    \mathbb E\abs{\lambda_s^{\psi+h}-\lambda_s^{\varphi+h}}\dd s
    +\int_0^T
    \mathbb E\abs{\lambda_s^\psi-\lambda_s^\varphi}\dd s.
\end{align*}
Applying \eqref{eq:forcing-stability} to the pairs
$(\psi+h,\varphi+h)$ and $(\psi,\varphi)$ yields
\begin{equation*}
  \abs{\Gamma_T(\psi)-\Gamma_T(\varphi)}
  \le\frac{2L}{1-\rho}
       \int_0^T\abs{\psi(s)-\varphi(s)}\dd s.
\end{equation*}
This completes the proof.
\end{proof}

\subsection{Local approximation of the bracket observable}

Applying the Level-3 upper bound to the bracket requires the empirical
functional $Q\mapsto\mathsf Q(Q)$ to be continuous on the relevant entropy
level sets.  Although $q$ is unbounded and depends on the entire past, the
response stability and the preceding tail estimates permit approximation by
bounded continuous local observables.

\begin{proposition}[Admissibility of the bracket observable]
\label{prop:bracket-admissible}
Under \Cref{ass:kernel,ass:rate}, the map $\mathsf Q:\cM_S\to\R$ in
\eqref{eq:Fq} is continuous on every level set
$\{Q:\mathsf H(Q)\le a\}$.  More precisely, there are bounded continuous
local observables $(q_n)$ such that, for every $a<\infty$,
\begin{equation}\label{eq:q-approx}
  \lim_{n\to\infty}
  \sup_{\mathsf H(Q)\le a}
  \int_\Omega\abs{q-q_n}\dd Q=0.
\end{equation}
\end{proposition}

\begin{proof}
First replace the corrector by
\begin{equation}\label{eq:gT}
  g_T=\int_0^T P_s f\dd s.
\end{equation}
By the add-one coupling,
\begin{equation}\label{eq:Dg-time-tail}
  \sup_{\omega\in\cX^-}\abs{Dg(\omega)-Dg_T(\omega)}
  \le\int_T^\infty r(s)\dd s
  \rightarrow 0,
\end{equation}
as $T\rightarrow\infty$.
This convergence is uniform and uses only $r\in L^1$.

We now truncate only the excitation inherited from the initial history.  Fix
$R>1$ and $0<\eta<1$, and choose a continuous function
$\chi_{R,\eta}:\R_+\to[0,1]$ that is zero on
$[0,\eta/2]\cup[R+1,\infty)$ and one on $[\eta,R]$, with linear
interpolation on the two remaining intervals.  For a two-sided configuration
$\omega$, define
\begin{align}
  z_{R,\eta}(\omega)
  &:=\int_{(-\infty,0)}\chi_{R,\eta}(-s)h(-s)\omega(\dd s),
    \label{eq:zReta}\\
  \psi_{R,\eta}(t,\omega)
  &:=\int_{(-\infty,0)}\chi_{R,\eta}(-s)h(t-s)\omega(\dd s),
    \qquad 0\le t\le T,\label{eq:psiReta}\\
  q_{K,R,T,\eta}(\omega)
  &:=\left[\phi\left(z_{R,\eta}(\omega)\right)
    \left(1+\Gamma_T(\psi_{R,\eta}(\cdot,\omega))\right)^2\right]\wedge K.
    \label{eq:qRTeta}
\end{align}
This observable depends only on the configuration in
$[-R-1,-\eta/2]$, and it is bounded by $K$.

We next verify continuity rather than appeal to endpoint smoothing.  If
$\omega_n\to\omega$ vaguely, then $z_{R,\eta}(\omega_n)\to
z_{R,\eta}(\omega)$, since its integrand is a continuous compactly supported
test function.  The family of test functions
\begin{equation*}
  s\longmapsto \chi_{R,\eta}(-s)h(t-s)\1_{\{s<0\}},
  \qquad 0\le t\le T,
\end{equation*}
is a compact subset of the continuous functions supported on
$[-R-1,-\eta/2]$.  Vague convergence therefore gives
\begin{equation*}
  \sup_{0\le t\le T}
  \abs{\psi_{R,\eta}(t,\omega_n)
       -\psi_{R,\eta}(t,\omega)}\longrightarrow0.
\end{equation*}
Together with \eqref{eq:Gamma-bounds} and continuity of $\phi$, this proves
that $q_{K,R,T,\eta}$ is a bounded continuous local observable.

It remains to prove uniform approximation.  Write
\begin{equation}\label{eq:full-past-excitation}
  \psi_\omega(t):=\int_{(-\infty,0)}h(t-s)\omega(\dd s).
\end{equation}
The future process from history $\omega$ has the same law as the Hawkes
process in \eqref{eq:forced-intensity} with forcing profile $\psi_\omega$.
Consequently,
\begin{equation}\label{eq:DgT-Gamma}
  Dg_T(\omega)=\Gamma_T(\psi_\omega).
\end{equation}
Set
\begin{align*}
  I_h^{R,\eta}(\omega)
  &:=\int_{(-\infty,0)}
    \left(1-\chi_{R,\eta}(-s)\right)h(-s)\omega(\dd s),\\
  I_H^{R,\eta}(\omega)
  &:=\int_{(-\infty,0)}
    \left(1-\chi_{R,\eta}(-s)\right)H(-s)\omega(\dd s).
\end{align*}
Tonelli's theorem yields
\begin{equation*}
  \int_0^T\abs{\psi_\omega(t)
       -\psi_{R,\eta}(t,\omega)}\dd t
  \le I_H^{R,\eta}(\omega).
\end{equation*}
Let $q^K:=q\wedge K$.  On the set $\{z(\omega)\le A\}$,
\eqref{eq:Dg-time-tail}, \eqref{eq:Gamma-bounds}, and the Lipschitz property of
$\phi$ imply, for a finite constant $C_A$ independent of
$K,R,T,\eta$ and $\omega$,
\begin{equation}\label{eq:qRT-pointwise-error}
  \abs{q^K(\omega)-q_{K,R,T,\eta}(\omega)}
  \le C_A\left\{
      \int_T^\infty r(s)\dd s
      +I_h^{R,\eta}(\omega)+I_H^{R,\eta}(\omega)
    \right\}.
\end{equation}
Outside that set both truncated observables are bounded by $K$.  Therefore
the right side of \eqref{eq:qRT-pointwise-error} may be augmented by
$K\1_{\{z>A\}}$ to obtain a bound valid on the whole state space.

Let $m_Q:=\mathbb{E}_Q[N(0,1]]$.  Campbell's formula
\citep{DaleyVereJones2003} and the definition of the cutoff give
\begin{align*}
  \int_\Omega I_h^{R,\eta}\dd Q
  &\le m_Q\left(\int_0^\eta h(s)\dd s+H(R)\right),\\
  \int_\Omega I_H^{R,\eta}\dd Q
  &\le m_Q\left(\int_0^\eta H(s)\dd s
                     +\int_R^\infty H(s)\dd s\right).
\end{align*}
Since $h$ is nonincreasing, the two integrals over $[0,\eta]$ are at
most $\eta h(0)$ and $\eta H(0)$, respectively.  Also, Campbell's formula
gives
\[
\int_\Omega z\dd Q=\norm h_1m_Q,
\]
so that Markov's inequality yields
\[
Q(z>A)\le\frac{\norm h_1m_Q}{A}.
\]
We have therefore proved
\begin{equation}\label{eq:qRT-error}
  \int_\Omega\abs{q^K-q_{K,R,T,\eta}}\dd Q
  \le \frac{K\norm h_1m_Q}{A}
    +C_A\left[\int_T^\infty r(s)\dd s
      +m_Q\left\{\eta+H(R)+\int_R^\infty H(s)\dd s\right\}\right].
\end{equation}
The level sets of $\mathsf H$ are compact in the strengthened topology, and
$Q\mapsto m_Q$ is continuous in that topology.  The quantities $m_Q$ are thus
uniformly bounded on each entropy level set.  Set
\begin{equation*}
  M_a:=\sup_{\mathsf H(Q)\le a}m_Q<\infty.
\end{equation*}
Sublinearity yields the remaining uniform-integrability result.  Indeed, for
every $\gamma>0$ there is $A_\gamma$ such that
$\phi(z)\le\gamma z$ for $z\ge A_\gamma$.  If
\begin{equation*}
  K>\frac{1}{(1-\rho)^2}
       \sup_{0\le z\le A_\gamma}\phi(z),
\end{equation*}
then \eqref{eq:q-bound} gives $\{q>K\}\subseteq\{z>A_\gamma\}$ and
\begin{equation*}
  0\le q-q^K
  \le q\1_{\{q>K\}}
  \le\frac{\gamma}{(1-\rho)^2}z.
\end{equation*}
Consequently,
\begin{equation*}
  \limsup_{K\to\infty}
  \sup_{\mathsf H(Q)\le a}
  \int_\Omega\left(q-q^K\right)\dd Q
  \le\frac{\gamma\norm h_1M_a}{(1-\rho)^2}.
\end{equation*}
Letting $\gamma\downarrow0$ proves
\begin{equation}\label{eq:qK-uniform-tail}
  \lim_{K\to\infty}
  \sup_{\mathsf H(Q)\le a}
  \int_\Omega\left(q-q^K\right)\dd Q=0.
\end{equation}
We now make the diagonal choice explicit.  For each $n\ge1$, first choose
$K_n\ge n$ so that
\begin{equation*}
  \sup_{\mathsf H(Q)\le n}
  \int_\Omega\left(q-q^{K_n}\right)\dd Q\le\frac{1}{2n}.
\end{equation*}
With $K_n$ fixed, choose successively $A_n,T_n,R_n\ge n$ and
$0<\eta_n\le1/n$ so that \eqref{eq:qRT-error} gives
\begin{equation*}
  \sup_{\mathsf H(Q)\le n}
  \int_\Omega
    \abs{q^{K_n}-q_{K_n,R_n,T_n,\eta_n}}\dd Q
  \le\frac{1}{2n}.
\end{equation*}
More precisely, $A_n$ first makes the term proportional to $K_n/A_n$
small; with $C_{A_n}$ then fixed, $T_n$ and $R_n$ make the two tail
terms small, and finally $\eta_n$ makes the remaining near-zero term small.
Set
\begin{equation}\label{eq:qn-diagonal-definition}
  q_n:=q_{K_n,R_n,T_n,\eta_n}.
\end{equation}
The triangle inequality yields
\begin{equation*}
  \sup_{\mathsf H(Q)\le n}
  \int_\Omega\abs{q-q_n}\dd Q\le\frac1n.
\end{equation*}
For every fixed $a<\infty$, the same bound holds on
$\{\mathsf H(Q)\le a\}$ whenever $n\ge a$.  This proves
\eqref{eq:q-approx}.  Since each $q_n$ is a bounded continuous local
observable, $\mathsf Q$ is a uniform limit, on every entropy level set, of
the continuous maps $Q\mapsto\int_\Omega q_n\dd Q$.  Hence $\mathsf Q$ is
continuous there.
\end{proof}

The next lemma removes the periodic extension in
\eqref{eq:periodized-empirical-process}.  This is not automatic for the
infinite-past observable $q$, so we use the quantitative approximation from
\Cref{prop:bracket-admissible}.

\begin{lemma}[Exponential removal of periodization]
\label{lem:remove-periodization}
For every $\delta>0$,
\begin{equation}\label{eq:remove-periodization}
  \limsup_{t\to\infty}\frac1t\log\mathbb{P}_\varnothing\left(
    \abs{\frac1t\int_0^t q(X_{s-})\dd s
      -\mathsf Q(\mathcal R_{t,N})}>\delta
  \right)=-\infty.
\end{equation}
\end{lemma}

\begin{proof}
For the diagonal sequence $(q_n)$ constructed in the proof of
\Cref{prop:bracket-admissible}, set
\begin{equation*}
  E_{t,n}
  :=\frac1t\int_0^t\abs{q(X_{s-})-q_n(X_{s-})}\dd s
   +\int_\Omega\abs{q-q_n}\dd\mathcal R_{t,N}.
\end{equation*}
We claim that, for every $\varepsilon>0$,
\begin{equation}\label{eq:q-exponentially-good}
  \lim_{n\to\infty}
  \limsup_{t\to\infty}\frac1t
  \log\mathbb{P}_\varnothing(E_{t,n}>\varepsilon)=-\infty.
\end{equation}
To see this, use \eqref{eq:q-tail-superexp} to replace $q$ by $q^K$.
For fixed $K$, augment \eqref{eq:qRT-pointwise-error} by
$K\1_{\{z>A\}}$, as in the proof of
\Cref{prop:bracket-admissible}, and let
\begin{equation*}
  w_{R,\eta}(a)
  :=\left(1-\chi_{R,\eta}(a)\right)(h(a)+H(a)).
\end{equation*}
Its mass satisfies
\begin{equation}\label{eq:w-mass}
  \norm{w_{R,\eta}}_1
  \le \eta\left(h(0)+H(0)\right)
      +H(R)+\int_R^\infty H(a)\dd a,
\end{equation}
which tends to zero by first sending $R\to\infty$ and then
$\eta\downarrow0$.  To identify each error explicitly, set
\begin{align*}
  B_{t,K}
  &:=\frac1t\int_0^t(q-q^K)(X_{s-})\dd s
    +\int_\Omega(q-q^K)\dd\mathcal R_{t,N},\\
  I_{t,A}
  &:=\frac1t\int_0^t
    \left(\1_{\{z_s>A\}}+\1_{\{z_s^{\mathrm{per}}>A\}}\right)\dd s,\\
  S_{t,R,\eta}
  &:=\frac{J_t(w_{R,\eta})+J_t^{\mathrm{per}}(w_{R,\eta})}{t},\\
  \tau_T&:=\int_T^\infty r(s)\dd s.
\end{align*}
For $q_n=q_{K,R,T,\eta}$, applying
\eqref{eq:qRT-pointwise-error} to the actual and periodized histories gives
\begin{equation}\label{eq:Etn-error-decomposition}
  E_{t,n}
  \le B_{t,K}+KI_{t,A}+C_AS_{t,R,\eta}+2C_A\tau_T.
\end{equation}
Here $KI_{t,A}$ is the indicator error, the two terms in
$S_{t,R,\eta}$ are respectively the actual and periodized shot-noise errors,
and $2C_A\tau_T$ is the deterministic corrector time-tail.  By
\eqref{eq:q-tail-superexp}, $B_{t,K}$ is superexponentially negligible as
$K\to\infty$.  For fixed $K$, \eqref{eq:z-tail-superexp} gives the same
conclusion for $KI_{t,A}$ as $A\to\infty$.  After $A$ is fixed, choose $T$
so large that $2C_A\tau_T$ is arbitrarily small.  Finally, choose
$R\to\infty$ and then $\eta\downarrow0$.  By \eqref{eq:w-mass},
$\norm{w_{R,\eta}}_1\to0$, so \eqref{eq:small-kernel-superexp} controls
$C_AS_{t,R,\eta}$.  Choosing these parameters diagonally as in
\eqref{eq:qn-diagonal-definition} proves
\eqref{eq:q-exponentially-good}.

By construction,
\begin{equation*}
  q_n=q_{K_n,R_n,T_n,\eta_n},
  \qquad 0\le q_n(\omega)\le K_n,
\end{equation*}
and
\begin{equation*}
  q_n(\omega)=q_n(\omega')
  \quad\text{whenever}\quad
  \omega|_{[-R_n-1,-\eta_n/2]}
  =\omega'|_{[-R_n-1,-\eta_n/2]}.
\end{equation*}
Thus, $K_n$ is a uniform bound for $q_n$, while $R_n+1$ is its maximal
past-window length.
If $t>R_n+1$, the actual and periodized histories coincide on this window
for $R_n+1\le s\le t$.  Hence,
\begin{equation*}
  \abs{\frac1t\int_0^t q_n(X_{s-})\dd s
       -\int_\Omega q_n\dd\mathcal R_{t,N}}
  \le\frac{2K_n(R_n+1)}{t}.
\end{equation*}
First let $t\to\infty$, then $n\to\infty$, and use
\eqref{eq:q-exponentially-good}.  The triangle inequality proves
\eqref{eq:remove-periodization}.
\end{proof}

\subsection{Bracket concentration}

The entropy rate has a unique zero.  Indeed, $\mathsf H(Q)=0$ implies that the
conditional intensity under $Q$ agrees with the nonlinear Hawkes intensity.
By uniqueness of the stationary dynamics under \eqref{eq:subcriticality}, this
forces $Q=\pi$.  Consequently, compactness of entropy level sets and
\Cref{prop:bracket-admissible} imply
\begin{equation}\label{eq:positive-gap}
  c_\delta:=
  \inf\left\{\mathsf H(Q):
  \abs{\mathsf Q(Q)-\mathsf Q(\pi)}\ge\delta\right\}>0,
  \qquad \delta>0.
\end{equation}

The positive entropy gap in \eqref{eq:positive-gap} assigns a strictly
positive large-deviation cost to every fixed departure of the empirical
bracket functional from its stationary value.  Combining this gap with the
exponential removal of periodization gives the uniform bracket stabilization
required by the martingale MDP as stated in the following proposition.

\begin{proposition}[Bracket concentration]\label{prop:bracket-concentration}
Let
\begin{equation}\label{eq:sigma-q}
  \sigma_g^2=\mathsf Q(\pi)=\mathbb{E}_\pi[q(X_{0-})].
\end{equation}
For every $T,\delta>0$,
\begin{equation}\label{eq:bracket-uniform-concentration}
  \limsup_{t\to\infty}\frac1t\log\mathbb{P}_\varnothing\left(
    \sup_{0\le u\le T}
    \abs{\frac1t\bracket{\widetilde M}_{tu}-u\sigma_g^2}>\delta
  \right)<0.
\end{equation}
\end{proposition}

\begin{proof}
We first prove the one-time estimate.  Fix $\delta>0$, let
\begin{equation*}
  \gamma_\delta:=\min\left\{1,c_{\delta/2}\right\}>0,
  \qquad a_\delta:=2\gamma_\delta,
\end{equation*}
and choose $n$ so large that \eqref{eq:q-approx} gives
\begin{equation}\label{eq:q-levelset-choice}
  \sup_{\mathsf H(Q)\le a_\delta}
  \abs{\mathsf Q(Q)-\int_\Omega q_n\dd Q}<\frac\delta4.
\end{equation}
The set
\begin{equation*}
  C_n
  =\left\{Q:\abs{\int_\Omega q_n\dd Q-\sigma_g^2}
        \ge\frac{3\delta}{4}\right\}
\end{equation*}
is closed in the strengthened topology.  If $Q\in C_n$ and
$\mathsf H(Q)\le a_\delta$, then \eqref{eq:q-levelset-choice} implies
$\abs{\mathsf Q(Q)-\sigma_g^2}\ge\delta/2$, and hence
$\mathsf H(Q)\ge c_{\delta/2}\ge\gamma_\delta$.  If
$\mathsf H(Q)>a_\delta$, the same lower bound is immediate.  Therefore
\begin{equation}\label{eq:local-observable-gap}
  \inf_{Q\in C_n}\mathsf H(Q)\ge\gamma_\delta.
\end{equation}
By \eqref{eq:empty-level3-upper},
\begin{equation*}
\limsup_{t\to\infty}\frac1t\log\mathbb{P}_\varnothing\left(
    \abs{\int_\Omega q_n\dd\mathcal R_{t,N}-\sigma_g^2}
       \ge\frac{3\delta}{4}\right)
  \le-\gamma_\delta.
\end{equation*}
Taking the parameters farther along the same approximation net if necessary,
\eqref{eq:q-exponentially-good} makes the probability that
$\int_\Omega\abs{q-q_n}\dd\mathcal R_{t,N}>\delta/4$ decay with any
prescribed exponential rate.  The union bound therefore gives
\begin{equation}\label{eq:periodized-one-time-bracket}
\limsup_{t\to\infty}\frac1t\log\mathbb{P}_\varnothing\left(
    \abs{\mathsf Q(\mathcal R_{t,N})-\sigma_g^2}>\delta
  \right)\le-\gamma_\delta.
\end{equation}
Apply \eqref{eq:periodized-one-time-bracket} with $\delta/2$, and use
\Cref{lem:remove-periodization} with the other half of the error.  We obtain
$\kappa_\delta=\gamma_{\delta/2}>0$ such that
\begin{equation}\label{eq:one-time-bracket}
  \limsup_{t\to\infty}\frac1t\log\mathbb{P}_\varnothing\left(
    \abs{\frac1t\int_0^t q(X_{s-})\dd s-\sigma_g^2}>\delta
  \right)\le-\kappa_\delta<0.
\end{equation}

Next, let us define
\begin{equation*}
  A_t(u):=\frac1t\bracket{\widetilde M}_{tu},
  \qquad 0\le u\le T.
\end{equation*}
Choose a finite grid
\begin{equation*}
  0=u_0<u_1<\cdots<u_m=T,
  \qquad
  \max_{0\le j<m}(u_{j+1}-u_j)\le\eta,
\end{equation*}
where $\eta\sigma_g^2<\delta/4$.  Since $A_t$ is nondecreasing, for
$u_j\le u\le u_{j+1}$,
\begin{align*}
  A_t(u)-u\sigma_g^2
  &\le
  A_t(u_{j+1})-u_{j+1}\sigma_g^2
  +(u_{j+1}-u)\sigma_g^2,\\
  u\sigma_g^2-A_t(u)
  &\le
  u_j\sigma_g^2-A_t(u_j)
  +(u-u_j)\sigma_g^2.
\end{align*}
It follows that
\begin{equation*}
  \bigcap_{j=1}^m
  \left\{
    \abs{A_t(u_j)-u_j\sigma_g^2}\le\frac\delta2
  \right\}
  \subseteq
  \left\{
    \sup_{0\le u\le T}
    \abs{A_t(u)-u\sigma_g^2}<\delta
  \right\}.
\end{equation*}
Therefore,
\begin{equation*}
  \mathbb P_\varnothing\left(
    \sup_{0\le u\le T}
    \abs{A_t(u)-u\sigma_g^2}>\delta
  \right)
  \le
  \sum_{j=1}^m
  \mathbb P_\varnothing\left(
    \abs{A_t(u_j)-u_j\sigma_g^2}>\frac\delta2
  \right).
\end{equation*}
For every $u_j>0$,
\begin{equation*}
  A_t(u_j)-u_j\sigma_g^2
  =u_j\left(
    \frac1{tu_j}\bracket{\widetilde M}_{tu_j}-\sigma_g^2
  \right).
\end{equation*}
Applying \eqref{eq:one-time-bracket} with time $tu_j$ and deviation
$\delta/(2u_j)$ gives
\begin{equation*}
  \limsup_{t\to\infty}\frac1t
  \log\mathbb P_\varnothing\left(
    \abs{A_t(u_j)-u_j\sigma_g^2}>\frac\delta2
  \right)
  \le-u_j\kappa_{\delta/(2u_j)}<0.
\end{equation*}
Since the grid is finite, the union bound yields
\begin{equation*}
  \limsup_{t\to\infty}\frac1t
  \log\mathbb P_\varnothing\left(
    \sup_{0\le u\le T}
    \abs{A_t(u)-u\sigma_g^2}>\delta
  \right)
  \le
  -\min_{1\le j\le m}u_j\kappa_{\delta/(2u_j)}<0,
\end{equation*}
which proves \eqref{eq:bracket-uniform-concentration}.
\end{proof}

\subsection{Proof of the main results}\label{sec:proof-mdp}

We first prove the variance formula, thereby identifying $\sigma_g^2$ with
the long-run variance $\sigma^2$ used in the martingale MDP.

\begin{proof}[Proof of \Cref{thm:variance}]
The response bound is \Cref{prop:corrector-bounds}.  The corrector boundary
is negligible on the central limit scale.  Indeed,
\Cref{prop:corrector-bounds} gives
\begin{equation}\label{eq:boundary-L2}
  \sup_{t\ge0}\mathbb{E}_\pi\left[\left(g(X_t)\right)^2\right]<\infty.
\end{equation}
Since $\widetilde M$ is square integrable,
\begin{equation*}
  \mathbb{E}_\pi\left[\widetilde M_t^2\right]
  =\mathbb{E}_\pi\left[\bracket{\widetilde M}_t\right]
  =t\sigma_g^2.
\end{equation*}
It follows from \eqref{eq:main-decomposition}, the Cauchy--Schwarz inequality, and
\eqref{eq:boundary-L2} that
\begin{equation}\label{eq:variance-limit}
\lim_{t\to\infty}\frac1t\Var_\pi(N_t)=\sigma_g^2.
\end{equation}
Thus $\sigma_g^2$ is the asymptotic variance $\sigma^2$; it agrees with the
variance coefficient in the functional CLT in \citep{Zhu2013CLT}.  This proves
\eqref{eq:variance-formula}.  Since $Dg\ge0$,
\begin{equation*}
  \sigma^2
  =\mathbb{E}_\pi\left[
      \lambda_0\left(1+Dg(X_{0-})\right)^2
    \right]
  \ge\mathbb{E}_\pi[\lambda_0]=\mu.
\end{equation*}
The inequality is strict under \eqref{eq:strict-condition}.  If
$h\not\equiv0$ and $\phi$ is strictly increasing on $\R_+$, inserting a point
creates a strictly positive direct intensity difference for a set of future
times of positive Lebesgue measure.  Hence $Dg(X_{0-})>0$ almost surely and
\eqref{eq:strict-condition} holds.
\end{proof}

\begin{proof}[Proof of \Cref{thm:main-mdp}]
We work first under the empty-history law.  By \eqref{eq:jump-bound},
$\widetilde M$ has jumps uniformly bounded by
$(1-\rho)^{-1}$.  By \Cref{prop:bracket-concentration}, its predictable
quadratic variation, after division by $t$, converges to $u\sigma^2$
uniformly on compact time intervals at an exponential rate with speed $t$.
These are precisely the hypotheses of the bounded-jump martingale functional MDP
in \citep[Proposition 1]{Dembo1996}, with $h_t=t$ in the notation of
\citep{Dembo1996} and
\begin{equation*}
  a_t=\frac{t}{b_t^2}.
\end{equation*}
Indeed, \eqref{eq:moderate-scale} is equivalent to $a_t\to0$ and
$ta_t\to\infty$, while \Cref{prop:bracket-concentration} is condition~(5) of
\citep[Proposition 1]{Dembo1996}.  See also \citep{Puhalskii1994,Gao1996}.  Hence,
\begin{equation*}
  \left\{\frac{\widetilde M_{tu}}{b_t}:0\le u\le1\right\}
\end{equation*}
satisfies an LDP with speed $b_t^2/t$ and rate function
\eqref{eq:rate-function} in the uniform topology.  Finally,
\eqref{eq:main-decomposition} and
\Cref{lem:boundary-negligible} show that the martingale and the centered
counting process are exponentially equivalent.  The exponential-equivalence theorem
\citep[Theorem 4.2.13]{DemboZeitouni2010} proves the MDP for the empty-history
process.

It remains to change the initial law.  By
\Cref{lem:stationary-empty-coupling}, the stationary and empty-history
processes can be coupled monotonically so that the total discrepancy
$D_\infty$ in \eqref{eq:total-coupling-discrepancy} has a finite exponential
moment.  Hence, for every
$\varepsilon>0$,
\begin{equation*}
  \mathbb{P}\left(\sup_{0\le u\le1}
    \abs{N^\pi(0,tu]-N^\varnothing(0,tu]}>\varepsilon b_t\right)
  \le Ce^{-c\varepsilon b_t},
\end{equation*}
which is superexponentially small at speed $b_t^2/t$, since
$t/b_t\to\infty$.  The stationary and empty-history centered paths are
exponentially equivalent, and a final application of the
exponential-equivalence theorem \citep[Theorem 4.2.13]{DemboZeitouni2010} proves \Cref{thm:main-mdp}.
\end{proof}

\begin{proof}[Proof of \Cref{cor:empty}]
The empty-history MDP, with the stationary centering $\mu t$, was established
as the first step in the proof of \Cref{thm:main-mdp}.
\end{proof}

\section{Conclusion}\label{sec:conclusion}

We have established a functional moderate deviation principle for
nonlinear Hawkes processes throughout the full range
$\sqrt t\ll b_t\ll t$, where the time $t\rightarrow\infty$.  The argument separates the problem into two parts.
The martingale-corrector reduction uses the subcritical response coupling to
produce a martingale with uniformly bounded jumps and an exponentially
negligible boundary.  The process-level LDP, together with the sublinear growth
of the rate, then gives exponential stabilization of the predictable quadratic
variation.  Coupling the empty-history and stationary processes completes the
passage between the initial laws.
The same corrector identifies the asymptotic variance through the add-one
response and shows that it dominates the stationary mean intensity in the
self-exciting case.  The proof also isolates the role of the assumptions:
Lipschitz subcriticality controls the response and corrector, while
sublinearity is needed for the superexponential high-intensity estimates.

\bibliographystyle{alpha}
\bibliography{bibtex}

\end{document}